\documentclass[a4paper,12pt]{article}
\usepackage{amsmath,amsthm,amssymb}
\usepackage[dvipdfmx]{graphicx}

\numberwithin{equation}{section}

\newtheorem{theorem}{Theorem}
\newtheorem{lemma}{Lemma}
\theoremstyle{definition}

\newtheorem{definition}{Definition}
\newtheorem*{rem}{{\bf Remark}}

\newcommand{\R}{\mathbb{R}}
\newcommand{\p}{\partial}
\newcommand{\eps}{\varepsilon}

\begin{document}

\begin{center}
\Large{\textbf{
Strong instability of standing waves \\
for nonlinear Schr\"{o}dinger equations \\
with double power nonlinearity}}
\end{center}

\vspace{5mm}

\begin{center}
{\large Masahito Ohta} \hspace{1mm} and \hspace{1mm}
{\large Takahiro  Yamaguchi}
\end{center}
\begin{center}
Department of Mathematics, 
Tokyo University of Science, \\
1-3 Kagurazaka, Shinjuku-ku, Tokyo 162-8601, Japan
\end{center}

\begin{abstract}
We prove strong instability (instability by blowup) 
of standing waves for some nonlinear Schr\"{o}dinger equations 
with double power nonlinearity. 
\end{abstract}

\section{Introduction}\label{sect:intro}

In this paper, 
we study instability of standing wave solutions $e^{i\omega t}\phi_{\omega}(x)$ 
for nonlinear Schr\"{o}dinger equations 
with double power nonlinearity: 
\begin{align}\label{nls}
i\p_t u=-\Delta u-a|u|^{p-1}u-b|u|^{q-1}u,
\quad (t,x)\in \R\times \R^N, 
\end{align}
where $a$ and $b$ are positive constants, 
$1<p<q<2^*-1$, $2^*=2N/(N-2)$ if $N\ge 3$, 
and $2^*=\infty$ if $N=1,2$. 

Moreover, we assume that $\omega>0$ and $\phi_{\omega}\in H^1(\R^N)$ is a ground state of 
\begin{align}\label{sp}
-\Delta \phi+\omega \phi-a|\phi|^{p-1}\phi-b|\phi|^{q-1}\phi=0, 
\quad x\in \R^N. 
\end{align}
For the definition of ground state, see \eqref{gs} below. 
It is well known that there exists a ground state $\phi_{\omega}$ of \eqref{sp} 
(see, e.g., \cite{BL,St}). 

The Cauchy problem for \eqref{nls} is locally well-posed in the energy space $H^1(\R^N)$
(see, e.g., \cite{caz,GV,Kato}). 
That is, for any $u_0\in H^1(\R^N)$ there exist $T^*=T^*(u_0)\in (0,\infty]$ 
and a unique solution $u\in C([0,T^*),H^1(\R^N))$ 
of \eqref{nls} with $u(0)=u_0$ such that 
either $T^*=\infty$ (global existence) 
or $T^*<\infty$ and $\displaystyle{\lim_{t\to T^*}\|\nabla u(t)\|_{L^2}=\infty}$ 
(finite time blowup). 

Furthermore, the solution $u(t)$ satisfies 
\begin{align}\label{conservation}
E(u(t))=E(u_0), \quad \|u(t)\|_{L^2}^2=\|u_0\|_{L^2}^2
\end{align}
for all $t\in [0,T^*)$, 
where the energy $E$ is defined by 
\begin{align*}
E(v)=\frac{1}{2}\|\nabla v\|_{L^2}^2
-\frac{a}{p+1}\|v\|_{L^{p+1}}^{p+1}
-\frac{b}{q+1}\|v\|_{L^{q+1}}^{q+1}.
\end{align*}

Here we give the definitions of stability and instability of standing waves. 

\begin{definition}
We say that the standing wave solution 
$e^{i\omega t}\phi_{\omega}$ of \eqref{nls} is {\em stable} 
if for any $\eps>0$ there exists $\delta>0$ such that 
if $\|u_0-\phi_{\omega}\|_{H^1}<\delta$, 
then the solution $u(t)$ of \eqref{nls} with $u(0)=u_0$ exists globally and satisfies 
$$\sup_{t\ge 0}\inf_{\theta\in \R,y\in \R^N}\|u(t)-e^{i\theta}\phi_{\omega}(\cdot+y)\|_{H^1}<\eps.$$ 
Otherwise, $e^{i\omega t}\phi_{\omega}$ is said to be {\em unstable}. 
\end{definition}

\begin{definition}
We say that $e^{i\omega t}\phi_{\omega}$ is {\em strongly unstable}  
if for any $\eps>0$ there exists $u_0\in H^1(\R^N)$ such that 
$\|u_0-\phi_{\omega}\|_{H^1}<\eps$ and 
the solution $u(t)$ of \eqref{nls} with $u(0)=u_0$ blows up in finite time.
\end{definition}

Before we consider the double power case, 
we recall some well-known results 
for the single power case: 
\begin{align}\label{nls1}
i\p_t u=-\Delta u-|u|^{p-1}u, 
\quad (t,x)\in \R\times \R^N. 
\end{align}
When $1<p<1+4/N$, 
the standing wave solution $e^{i\omega t}\phi_{\omega}$ of \eqref{nls1} is 
stable for all $\omega>0$ (see \cite{CL}). 
While, if $1+4/N\le p<2^*-1$, 
then $e^{i\omega t}\phi_{\omega}$ is strongly unstable for all $\omega>0$ 
(see \cite{BC} and also \cite{caz}). 

Next, we consider the double power case \eqref{nls} with $a>0$ and $b>0$. 
From Berestycki and Cazenave \cite{BC}, we see that 
if $1+4/N\le p<q<2^*-1$, 
then the standing wave solution $e^{i\omega t}\phi_{\omega}$ of \eqref{nls} 
is strongly unstable for all $\omega>0$ (see \cite{oht3} for the case $p=1+4/N<q$). 

On the other hand, 
when $1<p<1+4/N<q<2^*-1$, 
the standing wave solution $e^{i\omega t}\phi_{\omega}$ of \eqref{nls} 
is unstable for sufficiently large $\omega$ (see \cite{oht2}), 
while $e^{i\omega t}\phi_{\omega}$ is stable 
for sufficiently small $\omega$ (see \cite{fuk} 
and also \cite{oht1,mae} for more results in one dimensional case). 
However, it was not known 
whether $e^{i\omega t}\phi_{\omega}$ 
is strongly unstable or not 
for the case where $1<p<1+4/N<q<2^*-1$ and 
$\omega$ is sufficiently large. 

Now we state our main result in this paper. 

\begin{theorem}\label{thm1}
Let $a>0$, $b>0$, $1<p<1+4/N<q<2^*-1$, 
and let $\phi_{\omega}\in \mathcal{G}_{\omega}$. 
Then there exists $\omega_1>0$ such that 
the standing wave solution $e^{i\omega t}\phi_{\omega}$ of \eqref{nls} 
is strongly unstable for all $\omega\in (\omega_1,\infty)$. 
\end{theorem}

For $\omega>0$, 
we define functionals $S_{\omega}$ and $K_{\omega}$ on $H^1(\R^N)$ by
\begin{align*}
&S_{\omega}(v)=\frac{1}{2}\|\nabla v\|_{L^2}^2
+\frac{\omega}{2}\|v\|_{L^2}^2
-\frac{a}{p+1}\|v\|_{L^{p+1}}^{p+1}
-\frac{b}{q+1}\|v\|_{L^{q+1}}^{q+1}, \\
&K_{\omega}(v)=\|\nabla v\|_{L^2}^2
+\omega \|v\|_{L^2}^2
-a\|v\|_{L^{p+1}}^{p+1}
-b\|v\|_{L^{q+1}}^{q+1}.
\end{align*}
Note that  \eqref{sp} is equivalent to $S_{\omega}'(\phi)=0$, and 
\begin{align*}
K_{\omega}(v)=\p_{\lambda} S_{\omega}(\lambda v)\big|_{\lambda=1}
=\langle S_{\omega}'(v), v \rangle
\end{align*}
is the so-called Nehari functional. 
We denote the set of nontrivial solutions of \eqref{sp} by 
$$\mathcal{A}_{\omega}
=\{v\in H^1(\R^N): S_{\omega}'(v)=0, ~ v\ne 0\},$$ 
and define the set of ground states of \eqref{sp} by 
\begin{align}
\mathcal{G}_{\omega}
=\{\phi\in \mathcal{A}_{\omega}:
S_{\omega}(\phi)\le S_{\omega}(v)
\hspace{1mm} \mbox{for all} \hspace{1mm} 
v\in \mathcal{A}_{\omega}\}. 
\label{gs}
\end{align}
Moreover, consider the minimization problem:
\begin{align}
d(\omega)
=\inf\{S_{\omega}(v):v\in H^1(\R^N), ~ K_{\omega}(v)=0, ~ v\ne 0\}.
\label{gs1}
\end{align}
Then, it is well known that $\mathcal{G}_{\omega}$ is characterized as follows. 
\begin{align}
\mathcal{G}_{\omega}
=\{\phi\in H^1(\R^N): 
S_{\omega}(\phi)=d(\omega), ~ K_{\omega}(\phi)=0\}. 
\label{gs2} 
\end{align}

The proof of finite time blowup for \eqref{nls} relies on the virial identity \eqref{virial}. 
If $u_0\in \Sigma:=\{v\in H^1(\R^N): |x|v\in L^2(\R^N)\}$, 
then the solution $u(t)$ of \eqref{nls} with $u(0)=u_0$ 
belongs to $C([0,T^*),\Sigma)$, and satisfies 
\begin{align} \label{virial}
\frac{d^2}{dt^2}\|xu(t)\|_{L^2}^2=8P(u(t))
\end{align}
for all $t\in [0,T^*)$, where 
\begin{align*}
P(v)=\|\nabla v\|_{L^2}^2
-\frac{a\alpha}{p+1}\|v\|_{L^{p+1}}^{p+1}
-\frac{b\beta}{q+1}\|v\|_{L^{q+1}}^{q+1}
\end{align*}
with $\alpha=\dfrac{N}{2}(p-1)$, $\beta=\dfrac{N}{2}(q-1)$ 
(see, e.g., \cite{caz}). 

Note that for the scaling $v^{\lambda}(x)=\lambda^{N/2}v(\lambda x)$ for $\lambda>0$, 
we have 
\begin{align*}
&\|\nabla v^{\lambda}\|_{L^2}^2=\lambda^2\|\nabla v\|_{L^2}^2, ~ 
\|v^{\lambda}\|_{L^{p+1}}^{p+1}=\lambda^{\alpha}\|v\|_{L^{p+1}}^{p+1}, ~ 
\|v^{\lambda}\|_{L^{q+1}}^{q+1}=\lambda^{\beta}\|v\|_{L^{q+1}}^{q+1}, \\
&\|v^{\lambda}\|_{L^2}^2=\|v\|_{L^2}^2, \quad
P(v)=\p_{\lambda} E(v^{\lambda})\big|_{\lambda=1}. 
\end{align*}

The method of Berestycki and Cazenave \cite{BC} is based on the fact that 
$d(\omega)=S_{\omega}(\phi_{\omega})$ can be characterized as 
\begin{align}
d(\omega)
=\inf\{S_{\omega}(v):v\in H^1(\R^N), ~ P(v)=0, ~ v\ne 0\}
\label{bc}
\end{align}
for the case $1+4/N\le p<q<2^*-1$. 
Using this fact, it is proved in \cite{BC} that 
if $u_0\in \Sigma \cap \mathcal{B}_{\omega}^{BC}$ 
then the solution $u(t)$ of \eqref{nls} with $u(0)=u_0$ blows up in finite time, 
where 
$$\mathcal{B}_{\omega}^{BC}=
\{v\in H^1(\R^N):S_{\omega}(v)<d(\omega),~ P(v)<0\}.$$
We remark that \eqref{bc} 
does not hold for the case $1<p<1+4/N<q<2^*-1$. 

On the other hand, 
Zhang \cite{zhang} and Le Coz \cite{lec} gave an alternative proof of the result 
of Berestycki and Cazenave \cite{BC}. 
Instead of \eqref{bc}, they proved that  
\begin{align}
d(\omega)\le \inf\{S_{\omega}(v):
v\in H^1(\R^N), ~ P(v)=0, ~ K_{\omega}(v)<0\}
\label{zl}
\end{align}
holds for all $\omega>0$ if $1+4/N\le p<q<2^*-1$ 
(compare with Lemma \ref{lemma2} below). 
Using this fact, it is proved in \cite{zhang,lec} that 
if $u_0\in \Sigma \cap \mathcal{B}_{\omega}^{ZL}$ 
then the solution $u(t)$ of \eqref{nls} with $u(0)=u_0$ blows up in finite time, 
where 
$$\mathcal{B}_{\omega}^{ZL}=
\{v\in H^1(\R^N):S_{\omega}(v)<d(\omega),~ P(v)<0, ~ K_{\omega}(v)<0\}.$$

In this paper, we use and modify the idea of Zhang \cite{zhang} and Le Coz \cite{lec} 
to prove Theorem \ref{thm1}. 
For $\omega>0$ with $E(\phi_{\omega})>0$, 
we introduce 
\begin{align}
\mathcal{B}_{\omega}=\{v\in H^1(\R^N): 
0<E(v)<E(\phi_{\omega}), &~ 
\|v\|_{L^2}^2=\|\phi_{\omega}\|_{L^2}^2, 
\label{B} \\
&P(v)<0, ~ K_{\omega}(v)<0\}. 
\nonumber 
\end{align}
Then we have the following. 

\begin{theorem}\label{thm2}
Let $a>0$, $b>0$, $1<p<1+4/N<q<2^*-1$, 
and assume that $\phi_{\omega}\in \mathcal{G}_{\omega}$ satisfies $E(\phi_{\omega})>0$. 
If $u_0\in \Sigma \cap \mathcal{B}_{\omega}$, 
then the solution $u(t)$ of \eqref{nls} with $u(0)=u_0$ blows up in finite time. 
\end{theorem}

\begin{rem}
Our method is not restricted to the double power case \eqref{nls}, 
but is also applicable to other type of nonlinear Schr\"odinger equations. 
For example, we consider nonlinear Schr\"odinger equation with a delta function potential: 
\begin{align}\label{delta}
i\p_t u=-\p_x^2 u-\gamma \delta (x)u-|u|^{q-1}u,
\quad (t,x)\in \R\times \R, 
\end{align}
where $\delta(x)$ is the Dirac measure at the origin, $\gamma>0$ and $1<q<\infty$. 
The energy of \eqref{delta} is given by 
$$E(v)
=\frac{1}{2}\|\p_x v\|_{L^2}^2
-\frac{\gamma}{2} |v(0)|^2
-\frac{1}{q+1}\|v\|_{L^{q+1}}^{q+1}.$$
The standing wave solution  $e^{i\omega t}\phi_{\omega}(x)$ of \eqref{delta} 
exists for $\omega\in (\gamma^2/4,\infty)$. 

For the case $q>5$, 
it is proved in \cite{FOO} that 
there exists $\omega_2\in (\gamma^2/4,\infty)$ such that 
the standing wave solution $e^{i\omega t}\phi_{\omega}(x)$ of \eqref{delta} is stable 
for $\omega\in (\gamma^2/4,\omega_2)$, 
and it is unstable for $\omega\in (\omega_2,\infty)$. 
Since the graph of the function 
$$E(v^{\lambda})
=\frac{\lambda^2}{2}\|\p_x v\|_{L^2}^2
-\frac{\gamma \lambda}{2} |v(0)|^2
-\frac{\lambda^{\beta}}{q+1}\|v\|_{L^{q+1}}^{q+1}$$
with $\beta=\dfrac{q-1}{2}>2$ 
has the same properties as in Lemma \ref{lemma1} for \eqref{nls}, 
we can prove that 
the standing wave solution $e^{i\omega t}\phi_{\omega}(x)$ of \eqref{delta} is 
strongly unstable for $\omega$ satisfying $E(\phi_{\omega})>0$ 
(see also Theorem 5 of \cite{LFF} for the case $\gamma<0$). 
\end{rem}

The rest of the paper is organized as follows. 
In Section 2, we give the proof of Theorem \ref{thm2}. 
In Section 3, we show that $E(\phi_{\omega})>0$ for sufficiently large $\omega$, 
and prove Theorem \ref{thm1} using Theorem \ref{thm2}.

\section{Proof of Theorem \ref{thm2}}

Throughout this section, we assume that 
$$a>0, \quad b>0, \quad 1<p<1+4/N<q<2^*-1, \quad  E(\phi_{\omega})>0.$$
Recall that $0<\alpha=\dfrac{N}{2}(p-1)<2<\beta=\dfrac{N}{2}(q-1)$, and 
\begin{align}
&E(v^{\lambda})
=\frac{\lambda^2}{2}\|\nabla v\|_{L^2}^2
-\frac{a \lambda^{\alpha}}{p+1}\|v\|_{L^{p+1}}^{p+1}
-\frac{b \lambda^{\beta}}{q+1}\|v\|_{L^{q+1}}^{q+1}, 
\label{r1}\\
&P(v^{\lambda})
=\lambda^2 \|\nabla v\|_{L^2}^2
-\frac{a\alpha \lambda^{\alpha}}{p+1}\|v\|_{L^{p+1}}^{p+1}
-\frac{b\beta \lambda^{\beta}}{q+1}\|v\|_{L^{q+1}}^{q+1} 
=\lambda \p_{\lambda} E(v^{\lambda}), 
\label{r2} \\
&K_{\omega}(v^{\lambda})
=\lambda^2\|\nabla v\|_{L^2}^2
+\omega \|v\|_{L^2}^2
-\lambda^{\alpha}a\|v\|_{L^{p+1}}^{p+1}
-\lambda^{\beta}b\|v\|_{L^{q+1}}^{q+1}. 
\label{r3}
\end{align}

\begin{lemma}\label{lemma1}
If $v\in H^1(\R^N)$ satisfies $E(v)>0$, 
then there exist $\lambda_k=\lambda_k(v)$ $(k=1,2,3,4)$ such that 
$0<\lambda_1<\lambda_2<\lambda_3<\lambda_4$ and 
\begin{itemize}
\item $E(v^{\lambda})$ is decreasing in $(0,\lambda_1)\cup (\lambda_3,\infty)$, 
and increasing in $(\lambda_1,\lambda_3)$. 
\item $E(v^{\lambda})$ is negative in $(0,\lambda_2)\cup (\lambda_4,\infty)$, 
and positive in $(\lambda_2,\lambda_4)$. 
\item $E(v^{\lambda})<E(v^{\lambda_3})$ 
for all $\lambda\in (0,\lambda_3)\cup (\lambda_3,\infty)$.  
\end{itemize}
\end{lemma}

\begin{proof}
Since $a>0$, $b>0$, $0<\alpha<2<\beta$ and $E(v)>0$, 
the conclusion is easily verified by drawing the graph of \eqref{r1} 
(see Figure 1 below). 
\end{proof}

\centerline{\includegraphics[scale=0.8]{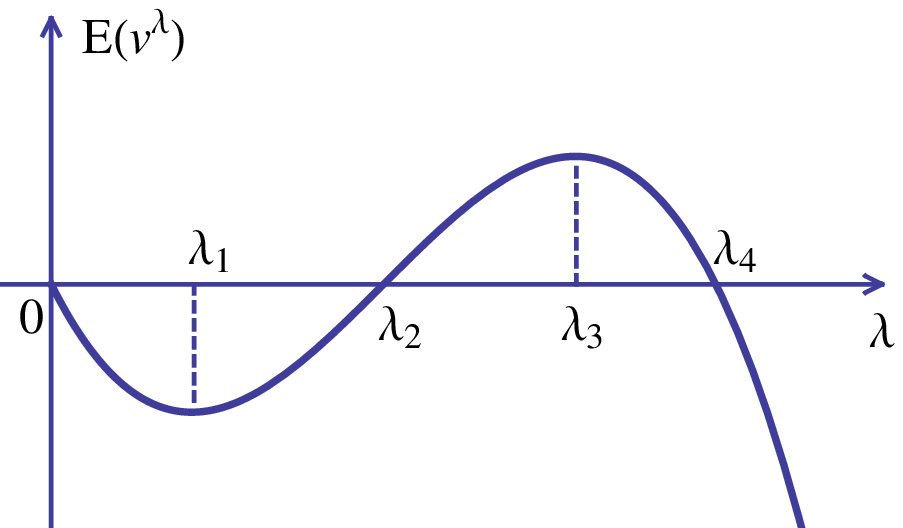} \hspace{5mm} 
Figure 1}

\begin{lemma}\label{lemma2}
If $v\in H^1(\R^N)$ satisfies 
$E(v)>0$, $K_{\omega}(v)<0$ and $P(v)=0$, 
then $d(\omega)<S_{\omega}(v)$. 
\end{lemma}

\begin{proof}
We consider two functions 
$f(\lambda)=K_{\omega}(v^{\lambda})$ and $g(\lambda)=E(v^{\lambda})$. 

Since $f(0)=\omega \|v\|_{L^2}^2>0$ and $f(1)=K_{\omega}(v)<0$, 
there exists $\lambda_0\in (0,1)$ such that 
$K_{\omega}(v^{\lambda_0})=0$. 
Moreover, since $v^{\lambda_0}\ne 0$, it follows from \eqref{gs1} that 
\begin{align*}
d(\omega)\le S_{\omega}(v^{\lambda_0}).
\end{align*}

On the other hand, 
since $g'(1)=P(v)=0$ and $g(1)=E(v)>0$, 
it follows from Lemma \ref{lemma1} that 
$\lambda_3=1$ and $g(\lambda)<g(1)$ for all $\lambda\in (0,1)$. 

Thus, we have 
$E(v^{\lambda_0})<E(v)$, and 
\begin{align*}
d(\omega)\le S_{\omega}(v^{\lambda_0}) 
=E(v^{\lambda_0})+\frac{\omega}{2}\|v^{\lambda_0}\|_{L^2}^2
<E(v)+\frac{\omega}{2}\|v\|_{L^2}^2=S_{\omega}(v). 
\end{align*}
This completes the proof. 
\end{proof}

\begin{lemma}\label{invariant}
The set $\mathcal{B}_{\omega}$ is invariant under the flow of \eqref{nls}. 
That is, if $u_0\in \mathcal{B}_{\omega}$, 
then the solution $u(t)$ of \eqref{nls} with $u(0)=u_0$ satisfies 
$u(t)\in \mathcal{B}_{\omega}$ for all $t\in [0,T^*)$. 
\end{lemma}

\begin{proof}
Let $u_0\in \mathcal{B}_{\omega}$ and 
let $u(t)$ be the solution of \eqref{nls} with $u(0)=u_0$. 
Then, by the conservation laws \eqref{conservation}, 
we have 
$$0<E(u(t))=E(u_0)<E(\phi_{\omega}), \quad 
\|u(t)\|_{L^2}^2=\|u_0\|_{L^2}^2=\|\phi_{\omega}\|_{L^2}^2$$
for all $t\in [0,T^*)$. 

Next, we prove that $K_{\omega}(u(t))<0$ for all $t\in [0,T^*)$. 
Suppose that this were not true. 
Then, since $K_{\omega}(u_0)<0$ and 
$t\mapsto K_{\omega}(u(t))$ is continuous on $[0,T^*)$, 
there exists $t_1\in (0,T^*)$ such that 
$K_{\omega}(u(t_1))=0$. 
Moreover, since $u(t_1)\ne 0$, by \eqref{gs1}, we have 
$d(\omega)\le S_{\omega}(u(t_1))$. 
Thus, we have 
\begin{align*}
d(\omega)\le S_{\omega}(u(t_1))
=E(u_0)+\frac{\omega}{2}\|u_0\|_{L^2}^2
<E(\phi_{\omega})+\frac{\omega}{2}\|\phi_{\omega}\|_{L^2}^2=d(\omega). 
\end{align*}
This is a contradiction. 
Therefore, $K_{\omega}(u(t))<0$ for all $t\in [0,T^*)$. 

Finally, we prove that $P(u(t))<0$ for all $t\in [0,T^*)$. 
Suppose that this were not true. 
Then, there exists $t_2\in (0,T^*)$ such that $P(u(t_2))=0$. 
Since $E(u(t_2))>0$ and $K_{\omega}(u(t_2))<0$, 
it follows from Lemma \ref{lemma2} that $d(\omega)<S_{\omega}(u(t_2))$. 
Thus, we have 
\begin{align*}
d(\omega)<S_{\omega}(u(t_2))
=E(u_0)+\frac{\omega}{2}\|u_0\|_{L^2}^2
<E(\phi_{\omega})+\frac{\omega}{2}\|\phi_{\omega}\|_{L^2}^2=d(\omega). 
\end{align*}
This is a contradiction. 
Therefore, $P(u(t))<0$ for all $t\in [0,T^*)$. 
\end{proof}

\begin{lemma}\label{EP}
For any $v\in \mathcal{B}_{\omega}$, 
$$E(\phi_{\omega})\le E(v)-P(v).$$
\end{lemma}

\begin{proof}
Since $K_{\omega}(v)<0$, as in the proof of Lemma \ref{lemma2}, 
there exists $\lambda_0\in (0,1)$ such that 
$S_{\omega}(\phi_{\omega})=d(\omega)\le S_{\omega}(v^{\lambda_0})$. 
Moreover, since $\|v^{\lambda_0}\|_{L^2}^2
=\|v\|_{L^2}^2=\|\phi_{\omega}\|_{L^2}^2$, 
we have 
\begin{equation}\label{s1}
E(\phi_{\omega})\le E(v^{\lambda_0}).
\end{equation}

On the other hand, 
since $P(v^{\lambda})=\lambda \p_{\lambda} E(v^{\lambda})$, 
$P(v)<0$ and $E(v)>0$, it follows from Lemma \ref{lemma1} that 
$\lambda_3<1<\lambda_4$. 
Moreover, since $\p_{\lambda}^2 E(v^{\lambda})<0$ for $\lambda\in [\lambda_3,\infty)$, 
by a Taylor expansion, we have 
\begin{equation}\label{s2}
E(v^{\lambda_3}) \le E(v)+(\lambda_3-1) P(v) \le E(v)-P(v).
\end{equation}

Finally, by \eqref{s1}, \eqref{s2} and the third property of Lemma \ref{lemma1}, 
we have 
\begin{align*}
E(\phi_{\omega})\le 
E(v^{\lambda_0})\le E(v^{\lambda_3}) \le E(v)-P(v).
\end{align*}
This completes the proof. 
\end{proof}

Now we give the proof of Theorem \ref{thm2}. 

\begin{proof}[Proof of Theorem \ref{thm2}]
Let $u_0\in \Sigma \cap \mathcal{B}_{\omega}$ 
and let $u(t)$ be the solution of \eqref{nls} with $u(0)=u_0$. 
Then, by Lemma \ref{invariant}, 
$u(t)\in \mathcal{B}_{\omega}$ for all $t\in [0,T^*)$. 

Moreover, by the virial identity \eqref{virial} and Lemma \ref{EP}, we have 
\begin{align*}
\frac{1}{8}\frac{d^2}{dt^2}\|xu(t)\|_{L^2}^2
=P(u(t))\le E(u(t))-E(\phi_{\omega})
=E(u_0)-E(\phi_{\omega})<0
\end{align*}
for all $t\in [0,T^*)$, 
which implies $T^*<\infty$. 
This completes the proof. 
\end{proof}

\section{Proof of Theorem \ref{thm1}}

First, we prove the following lemma. 

\begin{lemma}\label{lemma5}
Let $a>0$, $b>0$, $1<p<1+4/N<q<2^*-1$, 
and let $\phi_{\omega}\in \mathcal{G}_{\omega}$. 
Then there exists $\omega_1>0$ such that 
$E(\phi_{\omega})>0$ for all $\omega\in (\omega_1,\infty)$. 
\end{lemma}

\begin{proof}
Since $P(\phi_{\omega})=0$, 
we see that $E(\phi_{\omega})>0$ if and only if 
\begin{equation}\label{Epq}
\frac{(2-\alpha)a}{p+1}\|\phi_{\omega}\|_{L^{p+1}}^{p+1}
<\frac{(\beta-2)b}{q+1}\|\phi_{\omega}\|_{L^{q+1}}^{q+1}. 
\end{equation}
Moreover, in the same way as the proof of Theorem 2 in \cite{oht2}, 
we can prove that 
\begin{equation*}
\lim_{\omega\to \infty}
\frac{\|\phi_{\omega}\|_{L^{p+1}}^{p+1}}
{\|\phi_{\omega}\|_{L^{q+1}}^{q+1}}=0.
\end{equation*}
Thus, there exists  $\omega_1>0$ such that 
\eqref{Epq} holds for all $\omega\in (\omega_1,\infty)$. 
\end{proof}

\begin{proof}[Proof of Theorem \ref{thm1}]
Let $\omega\in (\omega_1,\infty)$. 
Then, by Lemma \ref{lemma5}, $E(\phi_{\omega})>0$. 

For $\lambda>0$, we consider the scaling 
$\phi_{\omega}^{\lambda}(x)=\lambda^{N/2}\phi_{\omega}(\lambda x)$, 
and prove that there exists $\lambda_0\in (1,\infty)$ such that 
$\phi_{\omega}^{\lambda}\in \mathcal{B}_{\omega}$ 
for all $\lambda\in (1,\lambda_0)$. 

First, we have $\|\phi_{\omega}^{\lambda}\|_{L^2}^2=\|\phi_{\omega}\|_{L^2}^2$ 
for all $\lambda>0$. 
Next, since $P(\phi_{\omega})=0$ and $E(\phi_{\omega})>0$, 
by Lemma \ref{lemma1} and \eqref{r2},  
there exists $\lambda_4>1$ such that 
$$0<E(\phi_{\omega}^{\lambda})<E(\phi_{\omega}), \quad 
P(\phi_{\omega}^{\lambda})<0$$ 
for all $\lambda\in (1,\lambda_4)$. 
Finally, since $P(\phi_{\omega})=0$, we have 
\begin{align*}
\p_{\lambda} K_{\omega}(\phi_{\omega}^{\lambda})\big|_{\lambda=1}
=-\frac{(p-1)a\alpha}{p+1}\|\phi_{\omega}\|_{L^{p+1}}^{p+1}
-\frac{(q-1)b\beta}{q+1}\|\phi_{\omega}\|_{L^{q+1}}^{q+1}<0. 
\end{align*}
Since $K_{\omega}(\phi_{\omega})=0$, 
there exists $\lambda_0\in (1,\lambda_4)$ such that 
$K_{\omega}(\phi_{\omega}^{\lambda})<0$ for all $\lambda\in (1,\lambda_0)$. 

Therefore, $\phi_{\omega}^{\lambda}\in \mathcal{B}_{\omega}$ 
for all $\lambda\in (1,\lambda_0)$. 
Moreover, since $\phi_{\omega}^{\lambda}\in \Sigma$ for $\lambda>0$,  
it follows from Theorem \ref{thm2} that for any $\lambda\in (1,\lambda_0)$, 
the solution $u(t)$ of \eqref{nls} with $u(0)=\phi_{\omega}^{\lambda}$ blows up 
in finite time. 

Finally, since $\displaystyle{\lim_{\lambda\to 1}
\|\phi_{\omega}^{\lambda}-\phi_{\omega}\|_{H^1}=0}$, 
the proof is completed. 
\end{proof}

\vspace{2mm} \noindent 
{\bf Acknowledgment}. 
The research of the first author was supported in part by 
JSPS KAKENHI Grant Number 24540163.

\end{document}